\documentclass[12pt]{article}
\usepackage{graphicx}
\usepackage{amsfonts,amssymb,amsthm,eucal,amsmath}
\usepackage{amsbsy}
\usepackage{amscd}
\usepackage{url}
\usepackage[inner=30mm, outer=30mm, textheight=225mm]{geometry}

\newtheorem{thm}{Theorem}[section]

\newtheorem{lemma}[thm]{Lemma}

\newtheorem{cor}[thm]{Corollary}
\theoremstyle{definition}
\newtheorem{defn}[thm]{Definition}

\newtheorem{ex}[thm]{Example}
\newtheorem{algor}[thm]{Algorithm}

\theoremstyle{remark}
\newtheorem{rmk}[thm]{Remark}

\newcommand{\co}{\colon\thinspace}

\def\H{\mathbb{H}}

\def\Z{\mathbb{Z}}
\def\tri{\mathcal{T}}

\title{Triangulations of hyperbolic 3-manifolds admitting strict angle structures} 

\author{Craig D. Hodgson, J. Hyam Rubinstein and Henry Segerman}

\begin{document}
\maketitle

\begin{abstract}
It is conjectured that every cusped hyperbolic 3-manifold has a decomposition into positive volume ideal hyperbolic tetrahedra (a ``geometric'' triangulation of the manifold). Under a mild homology assumption on the manifold we construct topological ideal triangulations which admit a strict angle structure, which is a necessary condition for the triangulation to be geometric. In particular, every knot or link complement in the 3-sphere has such a triangulation. We also give an example of a triangulation without a strict angle structure, where the obstruction is related to the homology hypothesis, and an example illustrating that the triangulations produced using our methods are not generally geometric. 
\end{abstract}

This work was supported by Australian Research Council grant DP1095760.

\section{Introduction}

Epstein and Penner \cite{epstein-penner} showed that every cusped\footnote{By \textbf{cusped}, we mean non-compact with finite volume.} hyperbolic 3-manifold has a 
decomposition into convex ideal polyhedra, which in the case of a manifold with one cusp is canonical. In many cases (for example punctured torus bundles and 2-bridge knot complements, see Gu\'eritaud and Futer \cite{gueritaud_torus_bundles}) the polyhedra of this canonical decomposition are tetrahedra, and we have a canonical ideal triangulation of the manifold in which every tetrahedron is positively oriented and so has positive volume (a {\bf geometric} ideal triangulation). Such a structure can be very useful and has been studied by many authors, starting with Thurston~\cite{thurston}. \\

However, in general the polyhedra may be more complicated than tetrahedra. The obvious approach to try to get a geometric ideal triangulation is to subdivide the polyhedra into tetrahedra. The difficulty is that the subdivision induces triangulations of the polygonal faces of the polyhedra, and these triangulations may not be consistent with each other. This can be fixed by inserting flat hyperbolic tetrahedra in between the two polyhedra, building a layered triangulation on the polygon that bridges between the two triangulations. The cost paid is the addition of the flat tetrahedra. The resulting triangulation is not geometric, and does not have some of the nice properties that a triangulation entirely consisting of positive volume tetrahedra has (see Choi~\cite{choi04} and Petronio and Weeks~\cite{petronio-weeks}).\\

Petronio and Porti~\cite{petronio-porti} discuss the problem of finding a geometric ideal triangulation, which remains unsolved. Nevertheless, it is commonly believed that every cusped hyperbolic 3-manifold has an ideal triangulation with positive volume tetrahedra, and experimental evidence from SnapPea~\cite{snappea} supports this. Wada, Yamashita and Yoshida \cite{han_yoshida,wada_yamashita_yoshida} have proved the existence of such triangulations given certain combinatorial conditions on the polyhedral decomposition, and Luo, Schleimer and Tillmann \cite{luo_schleimer_tillmann} show that such triangulations exist virtually.\\

In this paper we investigate an easier problem, that of finding an ideal triangulation with a strict angle structure. The existence of such a structure is a necessary condition for a geometric ideal triangulation. Our main result is:

\begin{thm}\label{no non periph H1 implies sas}
Assume that $M$ is a cusped hyperbolic 3-manifold homeomorphic to the interior of a compact 3-manifold $\overline{M}$ with torus or Klein bottle boundary components. If $H_1(\overline{M}; \Z_2) \to H_1(\overline{M}, \partial \overline{M}; \Z_2)$ is the zero map then $M$ admits an ideal triangulation with a strict angle structure. 
\end{thm}
\begin{cor}\label{link complements have sas}
If $M$ is a hyperbolic link complement in $S^3$, then $M$ admits an ideal triangulation with a strict angle structure. \end{cor}
\begin{proof}  For a link $L\subset S^3$, the peripheral elements generate $H_1(\overline{M})$, where $\overline{M}$ is the complement of a open regular neighbourhood of $L$ in $S^3$. 
This can be seen using a Mayer-Vietoris sequence, or just by observing that if we kill all of the meridian curves by filling in disks then we obtain $S^3$ minus a number of 3-balls, which has zero first homology. Therefore, the map to $H_1(\overline{M}, \partial \overline{M})$ is the zero map.
\end{proof}

Unfortunately, the triangulations we find will not generally be geometric (see Remark \ref{not_geometric} and Example \ref{strict angle struct not geometric}).\\

The idea of the construction is similar to the outline above of a method to find an ideal triangulation from the Epstein-Penner polyhedral decomposition: we carefully choose a subdivision of the polyhedra into ideal tetrahedra, and then insert flat tetrahedra to bridge between the identified faces of polyhedra that do not have matching induced triangulations. We cannot take our angle structure directly from the hyperbolic shapes of the resulting tetrahedra because of the inserted flat tetrahedra. One approach would be to try to deform the angle structure into a strict angle structure, opening out the flat tetrahedra so that each one has positive volume\footnote{Thinking along these lines, it is clear that we also have to be careful in how we insert the flat tetrahedra. For example, if we introduced an edge of degree 2, no strict angle structure would be possible.}. Instead of trying to deform the angle structure directly, we use work of Kang and Rubinstein~\cite{kang_rubinstein_2005} and Luo and Tillmann~\cite{luo_tillmann_2008} which relates the existence of a strict angle structure to the non-existence of certain ``vertical'' normal surface classes in the triangulation.  \\

This paper is organised as follows. In section \ref{sec: defns} we recall the standard definitions for ideal triangulations and angle structures. In section \ref{sec: triang constr} we introduce a framework for describing triangulations formed from polyhedral decompositions by subdividing the polyhedra and inserting flat tetrahedra into ``polygonal pillows'' to bridge between incompatible triangulations of the faces of the polyhedra. In section \ref{sec: detour}, we take a detour from the main thread of the paper, and give a direct combinatorial construction of an ideal triangulation which admits a strict angle structure for certain very special link complements. In section \ref{sec: layered triangs}, we give an algorithm for filling in a polygonal pillow with a layered triangulation, assuming that the polyhedra on either side are subdivided using a coning construction.
We then use this to give an algorithm that produces a triangulation given a polygonal decomposition. In section \ref{sec: dual normal classes}, we use the result of Kang-Rubinstein and Luo-Tillmann to link the existence of a strict angle structure to the non-existence of vertical normal classes. We also define the ``arc pattern'' for a normal class, which can roughly be thought of as the intersection of the normal class with the faces of the triangulation. In section \ref{sec: vert norm surf layered triang}, we use the way in which the arc pattern changes as we move through the layers of a layered triangulation to control the possible vertical normal classes.  We show that a vertical normal class can only have quadrilaterals in pillows with 4 or 6 sides. In Example \ref{4-gon in m136}, we give a triangulation with such a vertical normal class, with all quadrilaterals contained in a single tetrahedron (which can be seen as a layered triangulation of a 4-gonal pillow). This triangulation then has no strict angle structure, although it does have a semi-angle structure. In section \ref{sec: normal surfaces in PPP}, we give a homology condition on the manifold that rules out a vertical normal class, proving the main theorem. The strategy is to convert a vertical normal class into an embedded surface, in normal position relative to a cellulation of the manifold given by the decomposition into polyhedra and polygonal pillows. We illustrate the fact that our construction does not generally produce a geometric triangulation in Example \ref{strict angle struct not geometric}, with a triangulation which has a strict angle structure but that is not geometric. Finally, in section \ref{generalisations}, we discuss some possible generalisations of our results.

\section{Definitions}\label{sec: defns}

\subsection{Ideal triangulation}\label{sec:ideal triang}

Let $M$ be a topologically finite 3-manifold which is the interior of a compact 3-manifold with torus or Klein bottle boundary components. An ideal triangulation $\tri$ of $M$ consists of a pairwise disjoint union of standard Euclidean 3--simplices, $\widetilde{\Delta} = \cup_{k=1}^{n} \widetilde{\Delta}_k,$ together with a collection $\Phi$ of Euclidean isometries between the 2--simplices in $\widetilde{\Delta},$ called {\bf face pairings}, such that 
$(\widetilde{\Delta} \setminus \widetilde{\Delta}^{(0)} )/ \Phi$ is homeomorphic to $M.$
The simplices in $\tri$ may be singular. It is well-known that every non-compact, topologically finite 3--manifold admits an ideal triangulation. 

\subsection{Quadrilateral types}

Let $\Delta^3$ be the standard 3--simplex with a chosen orientation. Each pair of opposite edges corresponds to a normal isotopy class of quadrilateral disks in $\Delta^3,$ disjoint from the pair of edges. We call such an isotopy class a {\bf normal quadrilateral type}. Let $\tri^{(k)}$ be the set of all $k$--simplices in $\tri$. If $\sigma \in \tri^{(3)},$ then there is an orientation preserving map $\Delta^3 \to \sigma$ taking the $k$--simplices in $\Delta^3$ to elements of $\tri^{(k)},$ and which is a bijection between the sets of normal quadrilateral types. Let $\square$ denote the set of all normal quadrilateral types in $\tri.$\\

If $e\in \tri^{(1)}$ is any edge, then there is a sequence $(q_{n_1}, ..., q_{n_k})$ of normal quadrilateral types facing $e,$ which consists of all normal quadrilateral types dual to $e$ listed in sequence as one travels around $e.$ Then $k$ equals the degree of $e,$ and a normal quadrilateral type may appear at most twice in the sequence. This sequence is well-defined up to cyclic permutations and reversing the order. 

\subsection{Angle structures}

\begin{defn}[Generalised angle structure]\label{def:gen angle}

A function $\alpha\co \square \to \mathbb{R}$ is called a {\bf generalised angle structure on $(M,\tri)$} if it satisfies the following two properties:
\begin{enumerate}
\item If $\sigma^3 \in \tri^{(3)}$ and $q, q', q''$ are the three normal quadrilateral types supported by it, then
\begin{equation*}
   \alpha(q) + \alpha(q') + \alpha(q'') =\pi.
\end{equation*}
\item If $e\in \tri^{(1)}$ is any edge and $(q_{n_1}, ..., q_{n_k})$ is its normal quadrilateral type sequence, then
\begin{equation*}
\sum_{i=1}^k \alpha(q_{n_i})  =2\pi.
\end{equation*}
\end{enumerate}
\end{defn}

Dually, one can regard $\alpha$ as assigning angles $\alpha(q)$ to the two edges opposite $q$ in the tetrahedron containing $q$.\\

A generalised angle structure is called 
a {\bf semi-angle structure on $(M,\tri)$} if its image is contained in $[0,\pi],$ and 
a {\bf strict angle structure on $(M,\tri)$} if its image is contained in $(0,\pi).$

\section{Triangulations from polyhedral decompositions}\label{sec: triang constr}

\begin{defn}
In this paper, the term \textbf{polyhedron} will mean a combinatorial object obtained by removing all of the vertices from a 3-cell with a given combinatorial cell decomposition of its boundary. We further require that this can be realised as a positive volume convex ideal polyhedron in hyperbolic 3-space $\H^3$.
\end{defn}

\begin{defn}
A {\bf pyramid} is a polyhedron whose faces consist of an (ideal) $n$-gon and $n$ (ideal) triangles which form the cone of the boundary of the $n$-gon to a point. The polygon is called the {\bf base} of the pyramid, and the vertex not on the polygon is called the {\bf tip} of the pyramid.
\end{defn}

\begin{defn}
For a polyhedron $P$ and $v$ an (ideal) vertex of $P$, the {\bf coning of $P$ from $v$} is the decomposition of $P$ into pyramids whose tips are at $v$, and whose bases are the polygonal faces of $P$ not incident to $v$. 
\end{defn}

Note that a coning of $P$ from $v$ determines a triangulation of the faces of $P$ that are incident to $v$, but not of any other faces (unless, trivially, that face is a triangle itself). Note also that \emph{any} choice of triangulation of the base of a pyramid extends to a triangulation of the pyramid. See Figure \ref{polyhedron_with_coning.pdf}. 
\begin{figure}[htb]
\centering
\includegraphics[width=0.8\textwidth]{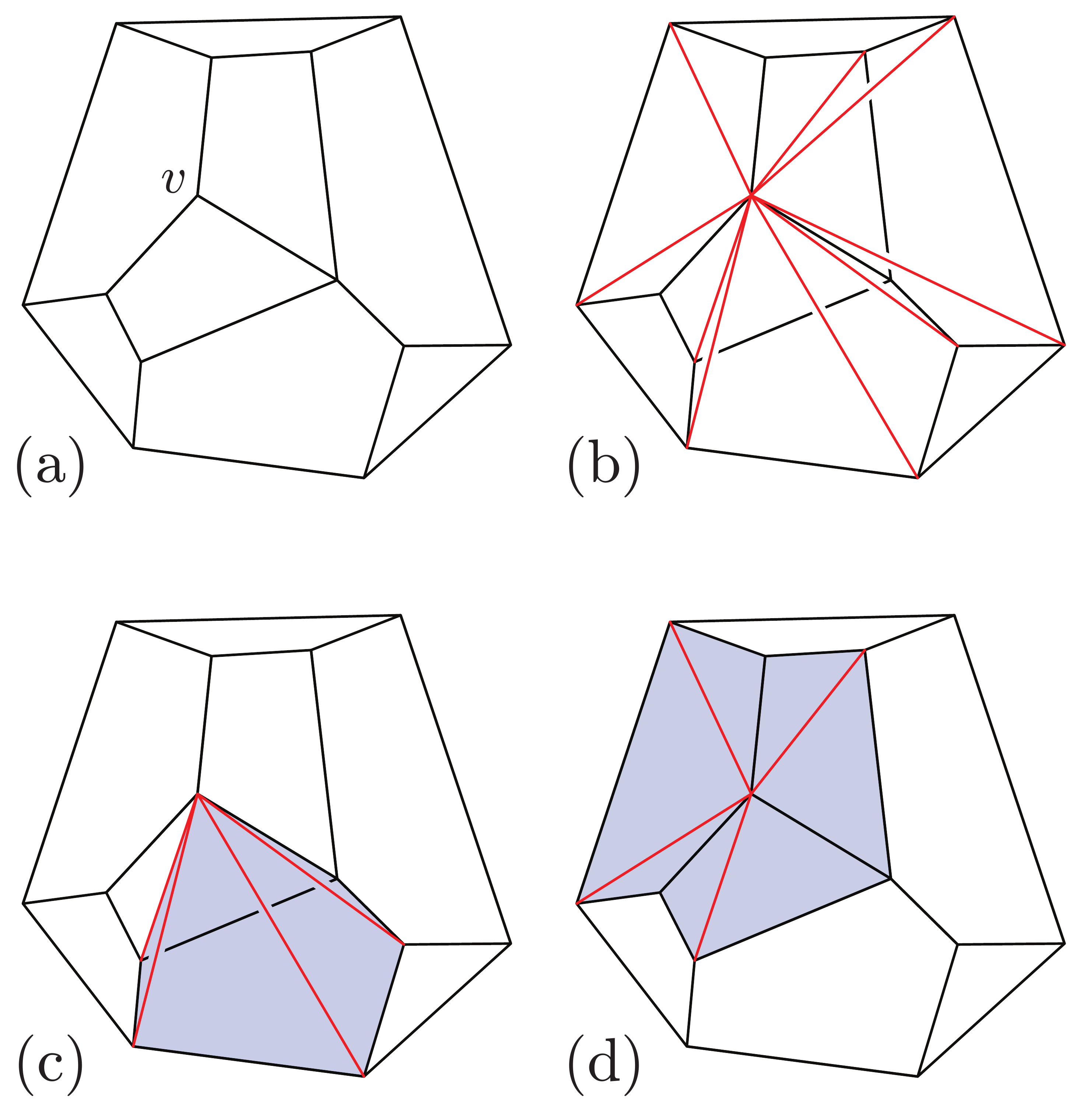}
\caption{In (a), a polyhedron $P$ with a choice of vertex $v$ (on the back of the polyhedron as viewed in the diagram). In (b) we add the edges corresponding to the coning of $P$ from $v$. In (c) a single pyramid is highlighted. In (d) we see that the faces of $P$ adjacent to $v$ have their triangulations determined by the coning, but not any of the other faces.}
\label{polyhedron_with_coning.pdf}
\end{figure}
\begin{defn}
An (ideal) {\bf polygonal pillow} or {\bf $n$-gonal pillow} is a combinatorial object obtained by removing all of the vertices from a 3-cell with a combinatorial cell decomposition of its boundary that has precisely two faces. The two faces are copies of an $n$-gon identified along corresponding edges.
\end{defn}
\begin{defn}
Suppose that $\mathcal{P}$ is a cellulation of a 3-manifold consisting of polyhedra 
and polygonal pillows with the property that polyhedra are glued to either polyhedra or polygonal pillows, but polygonal pillows are only glued to polyhedra. Then we call $\mathcal{P}$ a {\bf polyhedron and polygonal pillow cellulation}, or for short, a {\bf PPP-cellulation}. 
\end{defn}

\begin{defn}
Let $t$ be a triangulation of a polygon. A \textbf{diagonal flip} move changes $t$ as follows. First we remove an internal edge of $t$, producing a four sided polygon, one of whose diagonals is the removed edge. Second, we add in the other diagonal, cutting the polygon into two triangles and giving a new triangulation of the polygon.
\end{defn}

\begin{defn}
Let $Q$ be a polygonal pillow, with triangulations $t_-$ and $t_+$ given on its two polygonal faces $Q_-$ and $Q_+$. By a \textbf{layered triangulation} of $Q$, we mean a triangulation produced as follows. We are given a sequence of diagonal flips which convert $t_-$ into $t_+$. This gives a sequence of triangulations $t_-=L_1, L_2, \ldots, L_k=t_+$, where consecutive triangulations are related by a diagonal flip. Starting from $Q_-$ with the triangulation $t_-=L_1$, we glue a tetrahedron onto the triangulation $L_1$ so that two of its faces cover the faces of $L_1$ involved in the first diagonal flip. The other two faces together with the rest of $L_1$ produce the triangulation $L_2$. We continue in this fashion, adding one tetrahedron for each diagonal flip until we reach $L_k=t_+$, which we identify with $Q_+$.\\

We will refer to the triangulations $t_-=L_1, L_2, \ldots, L_k=t_+$ as the \textbf{layers} of the layered triangulation.
\end{defn}

Note that if $t_-$ and $t_+$ have shared edges, or even shared triangles, 
then we are abusing terminology  here since 
the resulting ``layered triangulation'' may not entirely `fill out' the polygonal pillow and give a genuine
triangulation of a topological polygonal pillow.

\begin{rmk}
Sleator, Tarjan and Thurston \cite{sleator_tarjan_thurston} have investigated triangulations of polygonal pillows, giving bounds on the number of diagonal flips required to change one triangulation of a polygon into another. 
\end{rmk}

\begin{defn}
Let $M$ be a cusped hyperbolic 3-manifold. A \textbf{geometric polyhedral decomposition} of $M$ is a decomposition of $M$ into positive volume convex ideal polyhedra in hyperbolic space $\H^3$.
\end{defn}
Note that every such $M$ has a geometric polyhedral decomposition, given by the Epstein-Penner decomposition \cite{epstein-penner}. Sakuma and Weeks \cite{sakumaweeks95} give some examples of Epstein-Penner decompositions. Examples of geometric polyhedral decompositions that are not necessarily canonical can be found in Thurston's notes \cite{thurston} and Aitchison-Reeves \cite{aitchison_reeves}.

\begin{rmk}\label{polyhedron triang dets PPP}
Suppose we are given a specified tetrahedral subdivision $\tri$ of the polyhedra of a geometric polyhedral decomposition $\mathcal{P}$ of a manifold. We think of the operation of inserting flat tetrahedra to bridge between incompatible triangulations of polygonal faces via an intermediate stage of using a PPP-cellulation. That is, given $(\tri, \mathcal{P})$, we produce a PPP-cellulation $\mathcal{P}'$, where $\mathcal{P}'$ is derived from $\mathcal{P}$, by inserting a polygonal pillow between two polyhedron faces if and only if their triangulations induced by $\tri$ do not match. We can then subdivide $\mathcal{P}'$, with each polyhedron being subdivided as in $\tri$ and some layered triangulation inserted into each polygonal pillow, to obtain a triangulation $\tri'$ of the whole manifold.
\end{rmk}
Note that this correspondence between pairs $(\tri, \mathcal{P})$ and PPP-cellulations $\mathcal{P}'$ applies no matter how we subdivide the polyhedra into tetrahedra; it is not specific to a coning subdivision.

\begin{rmk}[(Natural semi-angle structure)]  \label{rmk: natural angle structure}
If $P$ is a hyperbolic convex ideal polyhedron, as is the case for each polyhedron of a geometric polyhedral decomposition, then the tetrahedra obtained by coning $P$ from a vertex and then subdividing the resulting pyramids into tetrahedra using a triangulation of each pyramid's base all have positive volume. Now consider the polyhedra coming from a geometric polyhedral decomposition, with each polyhedron decomposed in this way. Add layered triangulations of polygonal pillows where the subdivisions do not agree. Then we can read off a semi-angle structure for the resulting triangulation using the ideal hyperbolic shapes for the tetrahedra inside the polyhedra, and with angles $0, 0$ and $\pi$ for each tetrahedron inside a polygonal pillow.
\end{rmk}

\section{A detour into alternating link complements}\label{sec: detour}

For certain links in $S^3$ we can construct an ideal triangulation which admits a strict angle structure in a relatively direct, combinatorial way. 

\begin{thm}\label{alternating link}
Suppose that $D$ is a reduced alternating diagram of a prime link $L \subset S^3$. Let $G \subset S^2$ be the 4-valent graph obtained by flattening the crossings of $D$. $G$ induces a decomposition of $S^2$ into faces, and we abuse notation by also referring to the resulting 2-complex as $G$. Let $G^*$ be the dual 2-complex. Suppose that $G$ has the following properties:
\begin{enumerate}
\item  \label{alt_link_hyp1}  $G$ has no bigons.
\item  \label{alt_link_hyp2} If a simple closed curve $C$ in $G^*$ intersects $G$ four times then one of the two components of $S^2 \setminus C$ contains a single vertex of $G$.
\item  \label{alt_link_hyp3} There exist vertices $v_1, v_2$ of $G$ so that no 2-cell of $G$ is adjacent to both $v_1$ and $v_2$.
\end{enumerate}
Then $S^3\setminus L$ has an ideal triangulation which admits a strict angle structure.
\end{thm} 
\begin{ex}
The Turk's head knot has a reduced alternating diagram with these properties. See Figure \ref{turks_head_knot.pdf}.
\end{ex}
\begin{figure}[htb]
\centering
\includegraphics[width=0.4\textwidth]{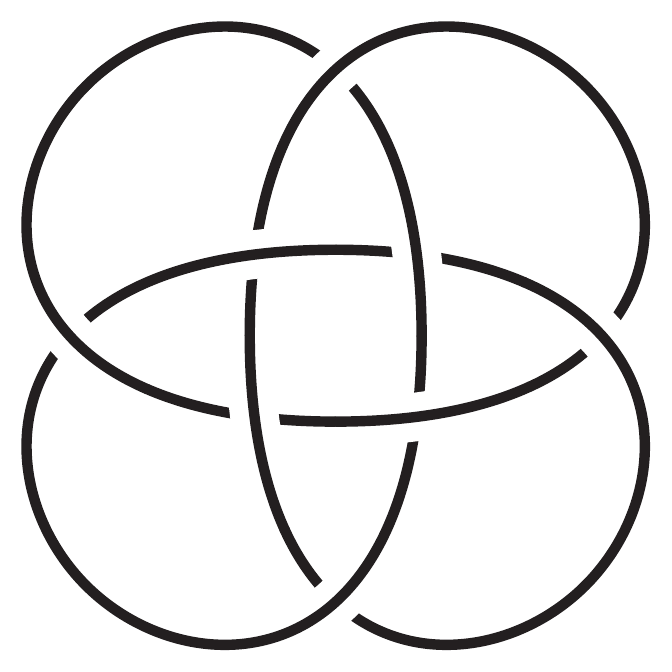}
\caption{The Turk's head knot, also known as $8_{18}$.}
\label{turks_head_knot.pdf}
\end{figure}

\begin{proof}[of Theorem \ref{alternating link}]
Consider the decomposition $\mathcal{P}$ of $S^3\setminus L$ into two polyhedra, as first studied by Thurston \cite{thurston}. Briefly, a reduced alternating, prime link complement can be decomposed into two polyhedra, one on either side of the 2-sphere containing the link projection.
See, for example, \cite{aitchison_lumsden_rubinstein} for details.\\

The resulting decomposition $\mathcal{P}$ consists of two polyhedra $P_1$ and $P_2$, each of which has boundary pattern given by $G \subset S^2$. Each face of $P_1$ is glued to the corresponding face of $P_2$, rotated ``one click'' either clockwise or anticlockwise. The direction of rotation alternates according to a checkerboard colouring of the complement of $G$ in $S^2$. Each edge of $\mathcal{P}$ is incident to $P_1$ twice and $P_2$ twice.\\

The first step is to find a set of dihedral angles for the two polyhedra so that the polyhedra can be realised as convex ideal hyperbolic polyhedra, and that the dihedral angles add to $2\pi$ around each edge. Note that the ideal hyperbolic polyhedra need not fit together nicely in $\mathbb{H}^3$.\\

We choose all of our dihedral angles to be $\pi/2$. Since all of the edge valences are four, the (interior) dihedral angles add up to $2\pi$. We next use the following theorem of Rivin \cite{rivin_ideal_hyp_polyhedra} (see also \cite{rivin_comb_opt,gueritaud_char_polyhedra}). 

\begin{thm}[(Rivin)]\label{rivin_andreev}
A combinatorial polyhedron $P$ can be realised as a convex ideal hyperbolic polyhedron with prescribed exterior angles assigned to the edges of $P$ if and only if
\begin{enumerate}
\item \label{rivin_cond1} Each exterior angle is in $(0,\pi)$. 
\item \label{rivin_cond2}  For each vertex $v$ of $P$, the sum of exterior dihedral angles for edges incident to $v$ is $2\pi$.
\item \label{rivin_cond3}  For each simple closed curve $C$ in the dual 2-complex $P^*$ that is not the boundary of a 2-cell of $P^*$, the sum of exterior dihedral angles for edges crossed by $C$ is strictly greater than $2\pi$. 
\end{enumerate}
\end{thm}

Our assignment of $\pi/2$ for each dihedral angle satisfies condition \ref{rivin_cond1} trivially, condition \ref{rivin_cond2} since each vertex of our polyhedra is valence 4, and condition \ref{rivin_cond3} by hypothesis \ref{alt_link_hyp2} of the statement of our theorem. Thus we can realise the polyhedra as (identical) convex ideal hyperbolic polyhedra. 
As in Remark \ref{rmk: natural angle structure}, any coning of the polyhedra will give a natural strict angle structure for the tetrahedra inside each polyhedron.\\

In the case given by hypothesis \ref{alt_link_hyp3} of the theorem, we can choose our cone vertices far enough apart from each other on $G$ so that the faces of $P_1$ and $P_2$ whose triangulations are determined by our conings are never glued to each other. Thus we can triangulate the pyramids of the two conings in such a way that they match up on the boundary of $P_1$ and $P_2$, and so that no flat tetrahedra need be inserted to bridge between the triangulations on each side. Then the natural strict angle structure on the tetrahedra of the polyhedra also gives us a strict angle structure on our constructed triangulation of the link complement.\end{proof}

\begin{rmk}
A similar argument also works for {\bf balanced} reduced alternating link diagrams. As defined in \cite{aitchison_lumsden_rubinstein}, these are such that every crossing point is a vertex of exactly one bigon. To obtain the two polyhedra, all of the bigons are collapsed. Every vertex of the resulting polyhedra has degree 3, and every edge in the cellulation has valence 6, so if we choose internal dihedral angles of $\pi/3$ then the above argument goes through.
\end{rmk}

\begin{rmk}
A beautiful class of examples is given by Aitchison-Reeves \cite{aitchison_reeves}. For the case of alternating links where the polyhedra can be given special regular hyperbolic structures, then the hyperbolic structure on the link complement can be seen by gluing two congruent copies of the same ideal hyperbolic polyhedron. These correspond to the Archimedean polyhedra. 
\end{rmk}

\section{Layered triangulations of polygonal pillows}\label{sec: layered triangs}

Given a PPP-cellulation of a 3-manifold $M$, we now specify a layered triangulation for each of our polygonal pillows, assuming that we subdivide each of our polyhedra by coning from some vertex of each, then choosing a triangulation for the bases of the resulting pyramids and subdividing the pyramids appropriately. (We will also say something about more general subdivisions in Section \ref{generalisations}.)\\

Given a cellulation of $M$ into polyhedra $\mathcal{P}$, the complex which is the union of the boundaries of the polyhedra is $\mathcal{P}^{(2)}$. After coning our polyhedra but before subdividing the pyramids, we identify two types of 2-cells in $\mathcal{P}^{(2)}$:
\begin{enumerate}
\item \label{2cell_type1} The triangulations of the faces of the polyhedra identified at the 2-cell are both specified by the coning process and they disagree. 
\item \label{2cell_type2} The triangulations of the faces of the polyhedra identified at the 2-cell are both specified by the coning process and they agree, or one or both is not specified.
\end{enumerate}

For the latter type of 2-cell, we can always choose the triangulation of any pyramids that are involved so that the triangulations of the polyhedra match at this pair of faces. Thus,
as in Remark \ref{polyhedron triang dets PPP}, we have an intermediate PPP-cellulation which has polygonal pillows only at the 2-cells of the first type. Now we need to specify a layered triangulation for each of these polygonal pillows. However, we see that the triangulations on either side of the polygonal pillow are of a very simple form: the polygons are triangulated in a coned fashion, with all edges that are internal to the polygon incident to a single cone vertex.\\

\begin{algor}(Layered triangulation between two cone triangulations of a polygon)\\

\label{layered triang cone triangulations}
\noindent {\bf Input:} two cone triangulations of a polygon.\\
 {\bf Output:} a layered triangulation of the polygonal pillow with the two given cone triangulations on each side.
\begin{enumerate}
\item If the two cone vertices are the same vertex of the polygon (or also if they are opposite vertices in the case that the polygon has 4 sides) then the two triangulations actually agree, and we are in case \ref{2cell_type2} above: there is no layered triangulation to make. So, assume that the two cone vertices are different. 
\item If the cone vertices are at distance greater than one, counting edges around the polygon, then the two triangulations share an internal edge: the edge $e$ between the two cone vertices. We will not change $e$, so our problem reduces to finding a layered triangulation for each of the two polygons obtained by cutting the original polygon at $e$. Each such polygon is either a triangle, which therefore needs no layered triangulation, or has cone triangulations on either side of the polygonal pillow, where the cone vertices are at distance one. 
\item Finally, suppose that the two cone vertices are at distance one from each other. Label the vertices of our polygon $v_1, v_2, \ldots, v_n$, in a cyclic order and suppose that we want to provide a sequence of diagonal flips changing the triangulation coned from $v_n$ into the triangulation coned from $v_1$. The sequence we use is to replace the edge $(v_n, v_i)$ with the edge $(v_1, v_{i+1})$, where $i$ runs from $2$ to $n-2$ in sequence. See Figure~\ref{retri_polygon_2_cones.pdf}.
\end{enumerate}
\end{algor}

\begin{figure}[htb]
\centering
\includegraphics[width=\textwidth]{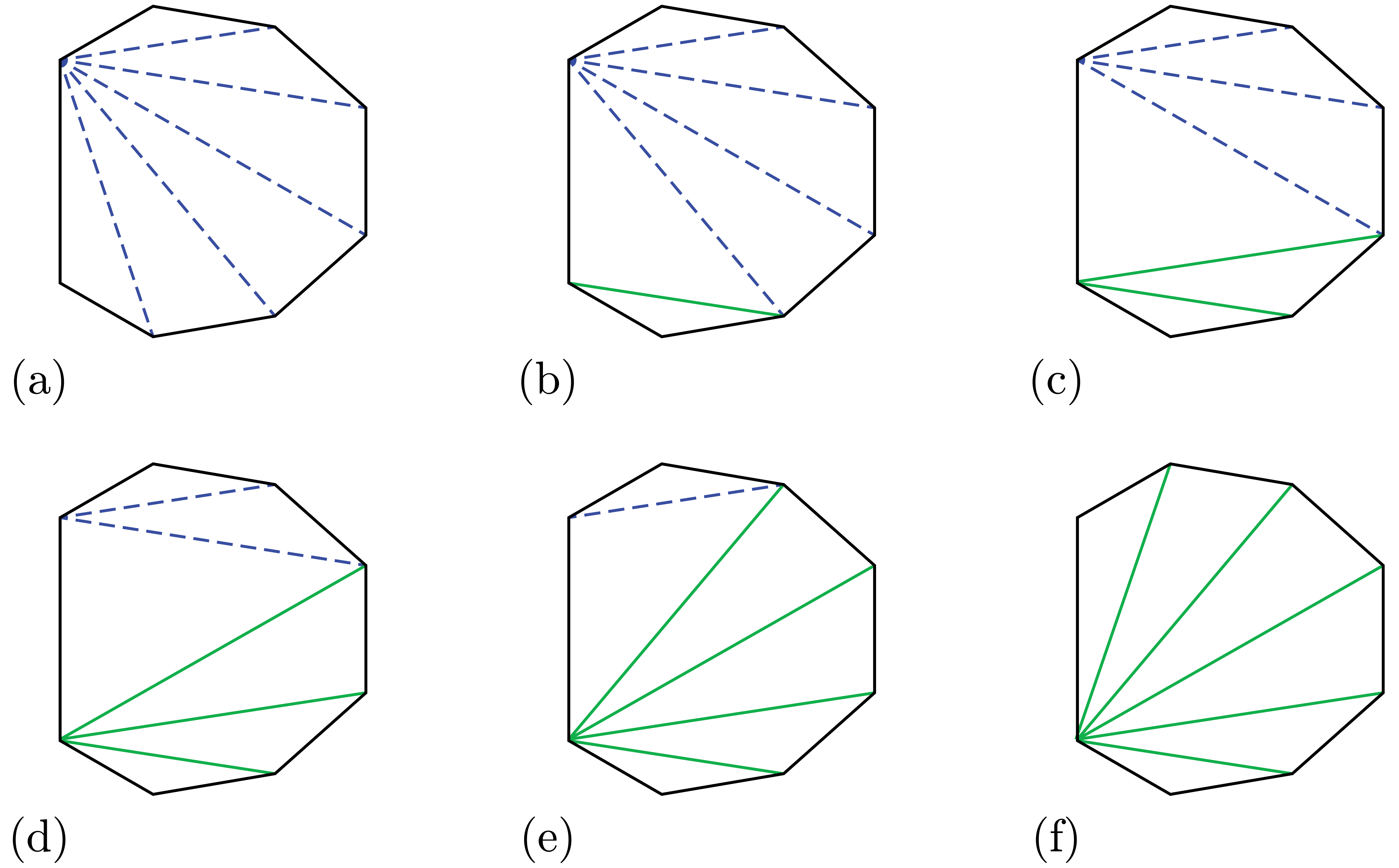}
\caption{Layered triangulation between two cone triangulations with cone vertices at distance one from each other.}
\label{retri_polygon_2_cones.pdf}
\end{figure}

We can then summarise the discussion of this section into a construction of a triangulation obtained by starting from a geometric polyhedral decomposition:

\begin{algor}(Triangulation from a geometric polyhedral decomposition)\\

\label{polyhedra cone vertices to triangulation}
\noindent {\bf Input:} A geometric polyhedral decomposition $\mathcal{P}$ of a 3-manifold $M$, with a choice of cone vertex $v_i$ for each polyhedron $P_i$ in $\mathcal{P}$.\\
 {\bf Output:} A PPP-cellulation $\mathcal{P}'$ of $M$ with the same polyhedra as $\mathcal{P}$, and a triangulation $\tri$ of $M$ consisting of a triangulation of each polyhedron of $\mathcal{P}'$ together with a layered triangulation of each polygonal pillow of $\mathcal{P}'$.
\begin{enumerate}
\item Cone each of the polyhedra of $\mathcal{P}$ as specified by the cone vertices. 
\item For each 2-cell of $\mathcal{P}^{(2)}$ for which the cellulation of the 2-cell induced by the conings of the polyhedra on either side disagree, insert a polygonal pillow. The resulting cellulation is $\mathcal{P}'$.
\item \label{step: ppp to triang} Triangulate each polygonal pillow using Algorithm \ref{layered triang cone triangulations}. Triangulate each pyramid within each polyhedron by choosing a triangulation of its base to match with the pattern on the corresponding face of the polyhedron glued to the other side of the base of the pyramid (this may involve a choice if the bases of two pyramids meet at a face, choose arbitrarily if so). The resulting triangulation is $\tri$.
\end{enumerate}
\end{algor}

\section{Dual normal classes}\label{sec: dual normal classes}

The main tool we will use to show the existence of strict angle structures is a result of Luo and Tillmann~\cite{luo_tillmann_2008}, based on work of Kang and Rubinstein~\cite{kang_rubinstein_2005}. This result links the existence of angle structures to 
normal surface theory 
using duality principles from linear programming. 
Given a $3$-manifold $M$ with an ideal triangulation $\tri$ as in Section \ref{sec:ideal triang},
the {\bf normal surface solution space}  $C(M; \tri)$ is a vector subspace of $\mathbb{R}^{7n},$ where $n$ is the number of tetrahedra in $\tri$, consisting of vectors  satisfying the {\bf compatibility equations} of normal surface theory. The coordinates of $x \in \mathbb{R}^{7n}$ represent weights of the four normal triangle types and the three normal quadrilateral types in each tetrahedron, and the compatibility equations state that normal triangles and quadrilaterals have to meet the 2--simplices of $\tri$ with compatible weights. \\

A vector in $\mathbb{R}^{7n}$ is called {\bf admissible} if at most one quadrilateral coordinate from each tetrahedron is non-zero and all coordinates are non-negative. An integral admissible element of $C(M; \tri)$ corresponds to a unique embedded, closed normal surface in $(M,\tri)$ and vice versa. As a reference for other facts from normal surface theory, please consult~\cite{jaco_oertel}.\\

There is a linear function $\chi^*\co C(M; \tri) \to \mathbb{R},$ which agrees with the Euler characteristic $\chi$ on embedded and immersed normal surfaces. If $\tri$ admits a generalised angle structure $\alpha$, then the formal Euler characteristic $\chi^*$ can be computed by
\begin{equation}\label{chi*}
2\pi \chi^*(x)= \sum_q -2\alpha(q) x_q,\end{equation}
where $x_q$ is the normal coordinate of the normal quadrilateral type $q$.\\

\begin{thm}[(Theorem 3 of \cite{luo_tillmann_2008})]\label{thm: duality}
Let $M$ be the interior of a compact 3-manifold with non-empty boundary and ideal triangulation $\tri$. Assume that each boundary component is a torus or a Klein bottle. Then the following are equivalent:
\begin{enumerate}
\item $(M;\tri)$ admits a strict angle structure.
\item for all $x\in C(M;\tri)$ with all quadrilateral coordinates nonÐnegative and at
least one quadrilateral coordinate positive, $\chi^*(x) < 0$.
\end{enumerate}
\end{thm}

As in Remark \ref{rmk: natural angle structure}, for a triangulation $\tri$ that is given by triangulating the polyhedra and polygonal pillows of a PPP-cellulation, where the PPP-cellulation comes from a geometric polyhedral decomposition, we have a natural semi-angle structure on $\tri$. 

\begin{defn}
A quadrilateral type $q$ in a tetrahedron is said to be {\bf vertical} (relative to a given semi-angle structure $\alpha$) if $\alpha(q) = 0$.
\end{defn}

\begin{cor}\label{cor: sas iff no vert classes}
The triangulation $\tri$ admits a strict angle structure if and only if there is no $x\in C(M;\tri)$ with all quadrilateral coordinates non-negative, all non-vertical quadrilateral coordinates zero and at
least one quadrilateral coordinate positive.
\end{cor}

\begin{proof}
This follows by combining Theorem \ref{thm: duality} with equation (\ref{chi*}). From the equation, if there is a quadrilateral type $q\in\square$
 such that $x_q>0$ and $\alpha(q)>0$, then $\chi^*(x) < 0$. Thus, we only have to worry about normal classes with no horizontal quadrilaterals.
\end{proof}

\begin{defn}
Let $\mathcal{N}(M;\tri) \subset C(M;\tri)$ be the subset of solutions to the normal surface compatibility equations such that every coordinate is a non-negative integer.
We refer to elements of $\mathcal{N}(M;\tri)$ as  \textbf{normal classes}.
\end{defn}

\begin{cor}\label{cor: sas iff no vert}
The triangulation $\tri$ admits a strict angle structure if and only if there is no $S\in \mathcal{N}(M;\tri)$ with all non-vertical quadrilateral coordinates zero and at least one quadrilateral coordinate positive.
\end{cor}

\begin{proof}
If we have an $x\in C(M;\tri)$ with the properties listed in Corollary \ref{cor: sas iff no vert classes}, then we can also find
such an $x$ for which all coordinates are integers. This follows since $x$ is a solution of a system of homogenous linear equations 
and inequalities with integer coefficients, so rational and hence integer solutions can also be found. 
In addition, we can also assume that the triangle coordinates are also non-negative integers. If this isn't already the case, we can add normal copies of the peripheral tori or Klein bottles (which are entirely made from normal triangles) until it is. After these modifications, $x$ is an element of $\mathcal{N}(M;\tri)$.\end{proof}

\begin{defn}
We refer to a normal class in $\mathcal{N}(M;\tri)$ with all non-vertical quadrilateral coordinates zero 
(as in Corollary \ref{cor: sas iff no vert}) as a 
\textbf{vertical normal class} or a \textbf{vertical class}. 
\end{defn}

\begin{rmk}\label{vertical only in pillows}
If we have a triangulation and semi-angle structure as constructed in Remark \ref{rmk: natural angle structure}, and we also have a normal class $S$ that is vertical relative to the semi-angle structure, then it can only have non-zero vertical quadrilaterals in the layered triangulations of the polygonal pillows, since the triangulations of the polyhedra have all angles positive. \end{rmk}

Next, we analyse the behaviour of a vertical class in relation to the layered triangulation.

\begin{defn}
Let $F$ be the number of triangles in $\tri$.
There is a well defined linear map $\alpha:\mathcal{N}(M;\tri) \to (\Z_{\geq0})^{3F}$, where $\alpha(S)$ counts the number of each type of normal arc in each of the triangles of $\tri$ induced by the normal class. For a given $S \in \mathcal{N}(M;\tri)$, these arcs can be realised as disjoint curves in each triangle, with the endpoints of the curves chosen so that they join up consistently at edges of the triangulation. We call these curves the \textbf{arc pattern} for the faces of $\tri$.
(This arc pattern is well-defined up to normal isotopy in $\tri^{(2)}$.)\\

Suppose $\mathcal{L}$ is a layered triangulation of a polygonal pillow within $\tri$, with layers $L_1, L_2, \ldots, L_n$. We refer to the restriction of the arc pattern to $L_i$ as the {\bf arc pattern} on the layer $L_i$.
\end{defn}

\begin{lemma} \label{change in arc pattern} Let $S \in \mathcal{N}(M;\tri)$ be a vertical normal class. Then
the arc patterns $A_i$ and $A_{i+1}$ of $S$ on consecutive layers $L_i$ and $L_{i+1}$ have the following properties:
\begin{enumerate}
\item  \label{prop1} Outside the quadrilateral on which $L_i$ and $L_{i+1}$ differ by a diagonal flip, $A_i$ is identical to $A_{i+1}$.
\item  \label{prop2} Inside the quadrilateral on which $L_i$ and $L_{i+1}$ differ by a diagonal flip, $A_i$ and $A_{i+1}$ are related as shown in Figure \ref{arc_pattern_changes.pdf}. Depending on the multiplicities of the normal arcs in the two triangles above and below the tetrahedron, the normal arcs connect to each other either as in diagram (a) or diagram (b). The multiplicities $a,b,c,d,e, a-\lambda, b+\lambda,c-\lambda, d+\lambda \geq 0$ are all non-negative integers, as is $\lambda$. If there are $x_1$ and $x_2$ of the two vertical quadrilateral types in the tetrahedron, then $\lambda = \min\{x_1,x_2\}$ and $e = |x_1 - x_2|$.
\end{enumerate}
\end{lemma}

\begin{figure}[htb]
\centering
\includegraphics[width=0.8\textwidth]{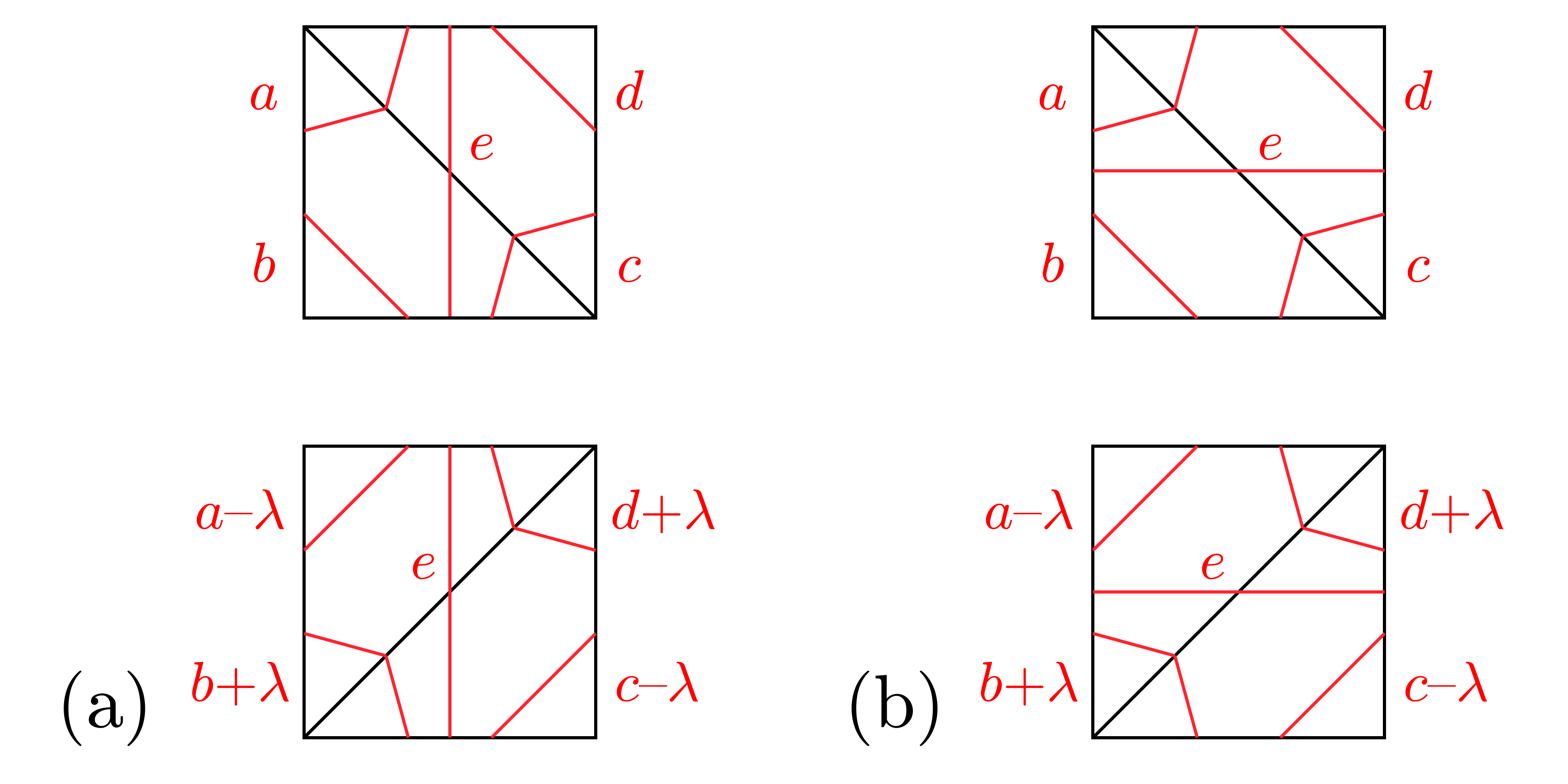}
\caption{The possible ways in which the arc pattern can change in the quadrilateral on which $L_i$ and $L_{i+1}$ differ by a diagonal flip.}
\label{arc_pattern_changes.pdf}
\end{figure}

\begin{proof}
Property \ref{prop1} is true because the triangles outside the quadrilateral do not change between the two layers. Property  \ref{prop2} follows by considering the normal arcs on the 4 faces of the tetrahedron induced by $x_1$ and $x_2$ vertical quadrilaterals of the two types and $t_i$ triangles at vertex $i$, as shown in Figure \ref{vertical_quads_and_tris_3_figs.pdf}.
These normal arcs fit together to give the arc pattern
shown in Figure \ref{arc_pattern_changes.pdf}(a) when $x_1 \ge x_2$,
and shown in Figure \ref{arc_pattern_changes.pdf}(b) when $x_1 \le x_2$, where
$\lambda = \min\{x_1,x_2\}$,
$a= t_1+\lambda, b=t_2, c=t_3+\lambda$, $d=t_4$ and $e = |x_1 - x_2|$.
\end{proof}

\begin{figure}[htb]
\centering
\includegraphics[width=\textwidth]{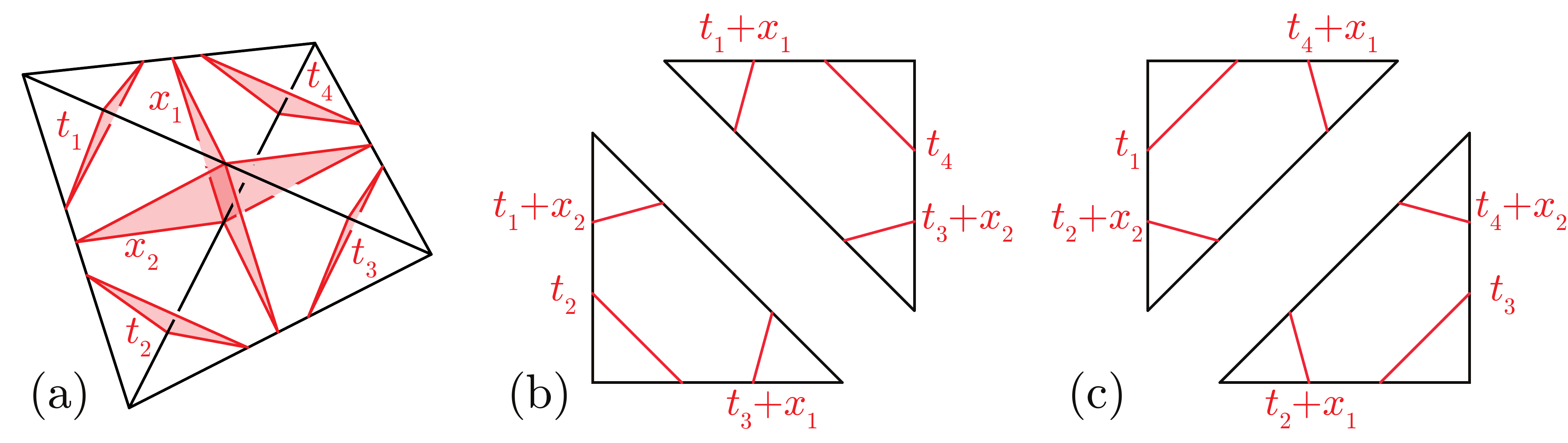}
\caption{In diagram (a), vertical quadrilaterals and triangles in a tetrahedron, with non-negative integer multiplicities $x_1,x_2$ and $t_1,t_2,t_3,t_4$ respectively. In diagrams (b) and (c), the numbers of normal arcs in each of the triangles on the top and the bottom of the tetrahedron, respectively.}
\label{vertical_quads_and_tris_3_figs.pdf}
\end{figure}

\begin{rmk}\label{arc_change_mnemonic}
A useful mnemonic for seeing how the arc pattern can change at a corner of a quadrilateral on which a tetrahedron is layered is as follows. Note that such a corner consists of either one or two corners of triangles of the triangulation of the layer. We change from one triangle corner to two, or vice versa, when we layer on the tetrahedron. If a quadrilateral corner ``increases'' from one to two triangle corners then it gains $\lambda \geq 0$ arcs, and if it ``decreases'' from two to one triangle corner then it loses  $\lambda \geq 0$ arcs.
\end{rmk}

\section{Vertical normal classes in a layered triangulation of a polygonal pillow}\label{sec: vert norm surf layered triang}

By Corollary \ref{cor: sas iff no vert} and Remark \ref{vertical only in pillows}, we are interested in vertical normal classes in the triangulation which have quadrilaterals only in the polygonal pillows. Therefore the only normal disks appearing in the polyhedra are triangles in each tetrahedron, which join up to form some number of parallel vertex linking disks at each vertex of the polyhedron. Thus, the arc pattern on each side of each polygonal pillow consists of some number of parallel vertex linking arcs at each vertex of the polygon. 

\begin{lemma}\label{change by +- lambda}
Let $S$ be a normal class in a triangulation of a PPP-cellulation $\mathcal{P}'$ of a manifold, and let $Q$ be an $n$-gonal pillow in $\mathcal{P}'$, with vertices labelled $v_1, \ldots, v_n$ (ordered cyclically).
Suppose that the arc patterns for $S$ on the top and bottom faces of $Q$ consist of
$w_i^+$ and  $w_i^-$ parallel vertex linking arcs at $v_i$ respectively, and no additional arcs.
 If $n$ is odd then $w_i^+ = w_i^-$ for all $i$. If $n$ is even then there is some $\lambda \in \Z$ such that $w_i^+ = w_i^- + (-1)^i \lambda$ for all $i$.
\end{lemma}

\begin{proof}
Let $a_i$ be the intersection number of $S$ with the edge between $v_i$ and $v_{i+1}$ (where we view the subscripts as being modulo $n$). Then $w_i^+  +  w_{i+1}^+ = a_i = w_i^-  +  w_{i+1}^-$. Thus $w_i^+  -  w_i^- = w_{i+1}^-  -  w_{i+1}^+$. If $n$ is odd then when we track the equations around the cycle we get that $w_i^+  -  w_i^- = w_i^-  -  w_i^+$, so $w_i^+ = w_i^-$. If $n$ is even then we see that the differences between the multiplicities on either side alternate, giving us $\lambda$ as in the statement of the Lemma.
\end{proof}

\begin{lemma}\label{quadrilaterals in pillow adjacent cone vertices}
Suppose we have a vertical normal class $S$ in a triangulation of a PPP-cellulation $\mathcal{P}$ of a manifold, and let $Q$ be an $n$-gonal pillow in $\mathcal{P}$. 
Assume that $Q$ has a layered triangulation produced as in Algorithm \ref{layered triang cone triangulations}, where the cone triangulations on either side have cone vertices at distance one from each other. Then $Q$ can contain vertical quadrilaterals of $S$ 
only if $n$ is 4.
\end{lemma}

\begin{proof}
We are in the case depicted in Figure \ref{retri_polygon_2_cones_arc_patterns.pdf}. We begin with some number of vertex linking arcs as shown in diagram (a). We will refer to the vertices of the polygon as $v_1$ through $v_n$, where the arc linking $v_i$ is labelled with multiplicity $w_i$ in diagram (a). Diagram (b) shows the arc pattern after the first diagonal flip, induced by layering a single tetrahedron onto the quadrilateral $Q_1$ with corners $v_1, v_2, v_3$ and $v_n$. Following Remark \ref{arc_change_mnemonic}, we see that the multiplicities at the four corners can go up or down by some $\lambda_1\geq0$ respectively. We have had to draw the arc pattern differently, with the new curve labelled with multiplicity $\lambda_1$ crossing the polygon. Note that despite this change, one can check that the combinatorics of the arc pattern does not change in the complementary polygon to $Q_1$ (as must be the case by property (i) of Lemma \ref{change in arc pattern}). One can also check that the arc multiplicities at the four corners of $Q_1$ are as given by property (ii) of Lemma \ref{change in arc pattern}.\\

We proceed to diagram (c) as before, layering a tetrahedron onto the quadrilateral $Q_2$ with corners $v_1, v_3, v_4$ and $v_n$. Here we have to draw the arc pattern in ``train track style'', in order to produce a picture in full generality. Again one can check that the multiplicities on arcs outside of $Q_2$ do not change, and that multiplicities at the corners of 
$Q_2$ change by $\lambda_2$ either up or down, as given by Remark \ref{arc_change_mnemonic}. This continues, layering a tetrahedron on each $Q_i$, with corners $v_1, v_{i+1}, v_{i+2}$ and $v_n$, 
up to the last quadrilateral, $Q_{n-3}$. 
At this stage we are in the situation of diagram (f). 
Here, all of the multiplicities of non-vertex linking arcs have to be zero, and the multiplicities of vertex linking arcs are 
given by Lemma \ref{change by +- lambda}. \\

In particular, consider the arc segment labelled $\sum_{i=2}^{n-3} \lambda_i$. This segment is made up of a number of arcs that follow the train tracks across the polygon, ending at possibly many different sides of the polygon. However, all of those arcs are non-vertex linking, so they must all have multiplicity zero. Since all $\lambda_i \geq 0$,  it follows that all $\lambda_i$ except possibly for $\lambda_1$ must be zero. Then the same argument applied to the arc labelled $\lambda_1 + \lambda_2$ shows that $\lambda_1 = 0$ also. 
By Lemma \ref{change in arc pattern}, 
this means that we have only vertex linking arcs at every layer of the layered triangulation, and that  these are connected to each other only by normal triangles. 
This argument applies unless there is no $\lambda_2$, which corresponds to the case $n=4$. 
\end{proof}

\begin{rmk} \label{quads in quad}
The $n=4$ case in the Lemma occurs for $\lambda >0$ as in Figure \ref{arc_pattern_changes.pdf} when $e=0$, i.e.\thinspace when the two vertical quadrilateral types occur with equal multiplicity $x_1=x_2=\lambda$.
\end{rmk}

\begin{figure}[htbp]
\centering
\includegraphics[width=\textwidth]{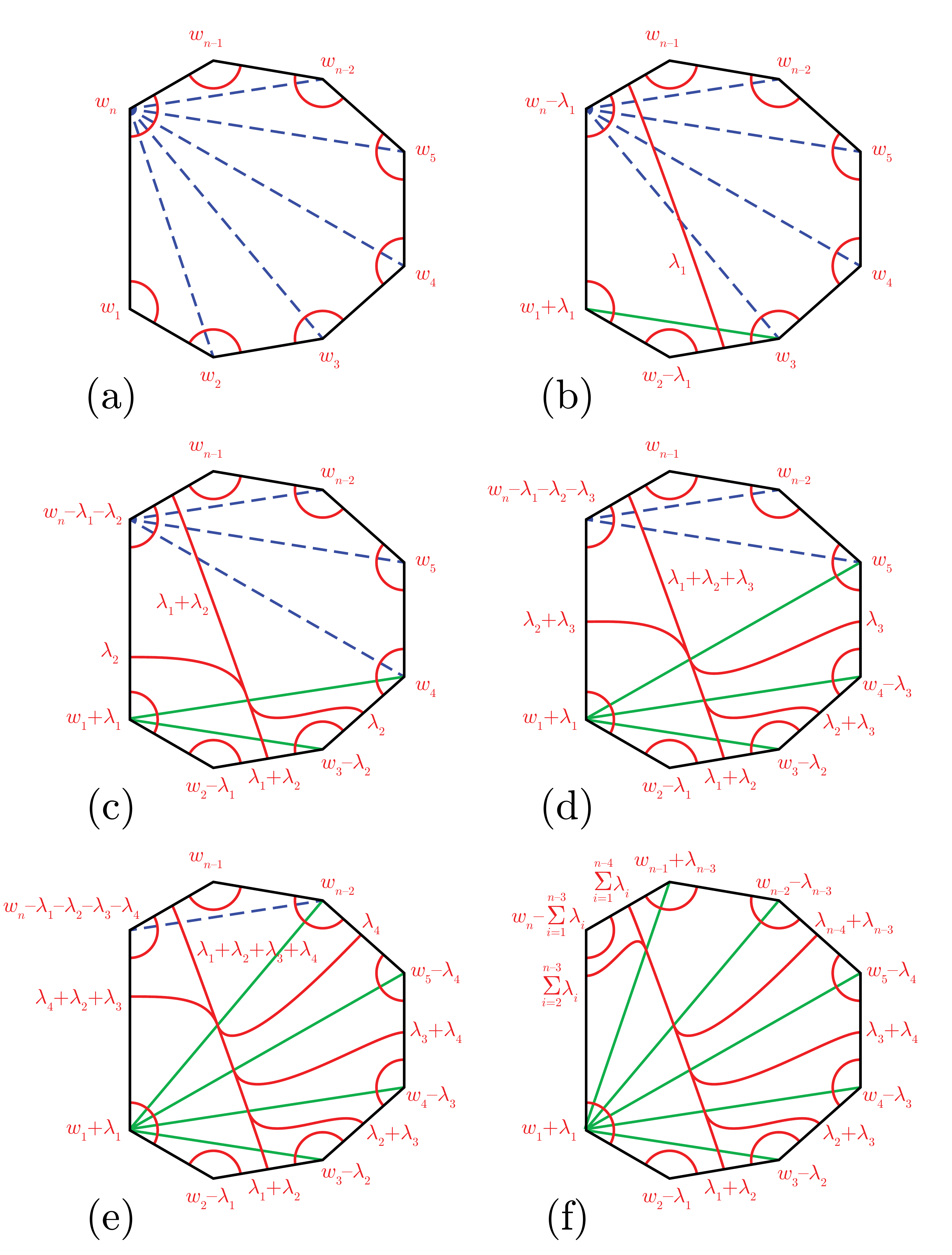}
\caption{Layered triangulation between two cone triangulations at distance one from each other, with the possible arc patterns that start from a pattern consisting of vertex linking arcs. The labels near to each arc refer to the multiplicity at that arc, with the multiplicities at unlabelled arcs implied by the smoothing at the ``train track junctions''.
All $w_i, \lambda_i$ and all multiplicities that appear are non-negative integers.}
\label{retri_polygon_2_cones_arc_patterns.pdf}
\end{figure}

\begin{lemma}\label{quadrilaterals in only 4 and 6 gons}
Suppose we have a vertical normal class $S$ in a triangulation of a PPP-cellulation $\mathcal{P}$ of a manifold, and let $Q$ be an $n$-gonal pillow in $\mathcal{P}$. Assume that $Q$ has a layered triangulation produced as in Algorithm \ref{layered triang cone triangulations}. Then $Q$ can contain vertical quadrilaterals of $S$ only if $n$ is 4 or 6.
\end{lemma}
\begin{proof}
With the notation of Algorithm \ref{layered triang cone triangulations}, suppose that the matching edge $e$ between the two cone vertices has endpoints at distance $k \leq n/2$ from each other around the $n$-gon. 
The triangulation consists of two layered triangulations, bridging between cone triangulations of two subpolygons, a ($k+1$)-gon and an ($n-k+1$)-gon, where the cone vertices in these subpolygons are at distance one from each other. If $k=1$, then one of these subpolygons is a trivial bigon.\\

By Lemma \ref{quadrilaterals in pillow adjacent cone vertices}, at least one subpolygon has 4 sides, and this 4-gon contains vertical quadrilaterals. If the other subpolygon is trivial then we are in the case $n=4$. Hence assume that the other subpolygon is non-trivial. Since the 4-gon contains vertical quadrilaterals, from Lemma \ref{change in arc pattern} we see that there must be a change in the multiplicities of the (necessarily vertex linking) arc pattern above and below the pillow. By Lemma \ref{change by +- lambda}, there is some $\lambda > 0$ such that the change in multiplicity around the pillow alternates by plus or minus $\lambda$. Thus, the multiplicity changes at the vertices of the pillow contained within the other subpolygon, and hence this other subpolygon must also be a 4-gon, by Lemma \ref{quadrilaterals in pillow adjacent cone vertices}. This is the case $n=6$.\end{proof}

\begin{rmk}
The case $n=4$ occurs as in Remark \ref{quads in quad}. 
The case $n=6$ occurs when the cone vertices are distance 3 from each other, giving two 4-gons as the subpolygons.
\end{rmk}

\begin{ex}\label{ex: 4-gon}
Suppose that we have a cellulation with a 4-gonal face
that has edge identifications as shown in Figure \ref{4-gon_example.pdf}(a). Assume that the polyhedra of the cellulation are triangulated so that the two induced triangulations of the 4-gon faces do not agree. Then we can put two vertical quadrilaterals in the tetrahedron as shown in Figure \ref{4-gon_example.pdf}(b). It is easy to check that the Q-matching equations hold (see \cite{tollefson} for details) for these two quadrilaterals. (Note that the directions of the arrows don't matter, only that the edges are identified in neighbouring pairs.) 
If in addition, triangles can be added to these quadrilaterals to form a closed normal class (as opposed to a spun-normal class), then we have the existence of a vertical normal class in $\mathcal{N}(M;\tri)$.
\end{ex}

\begin{figure}[htb]
\centering
\includegraphics[width=0.6\textwidth]{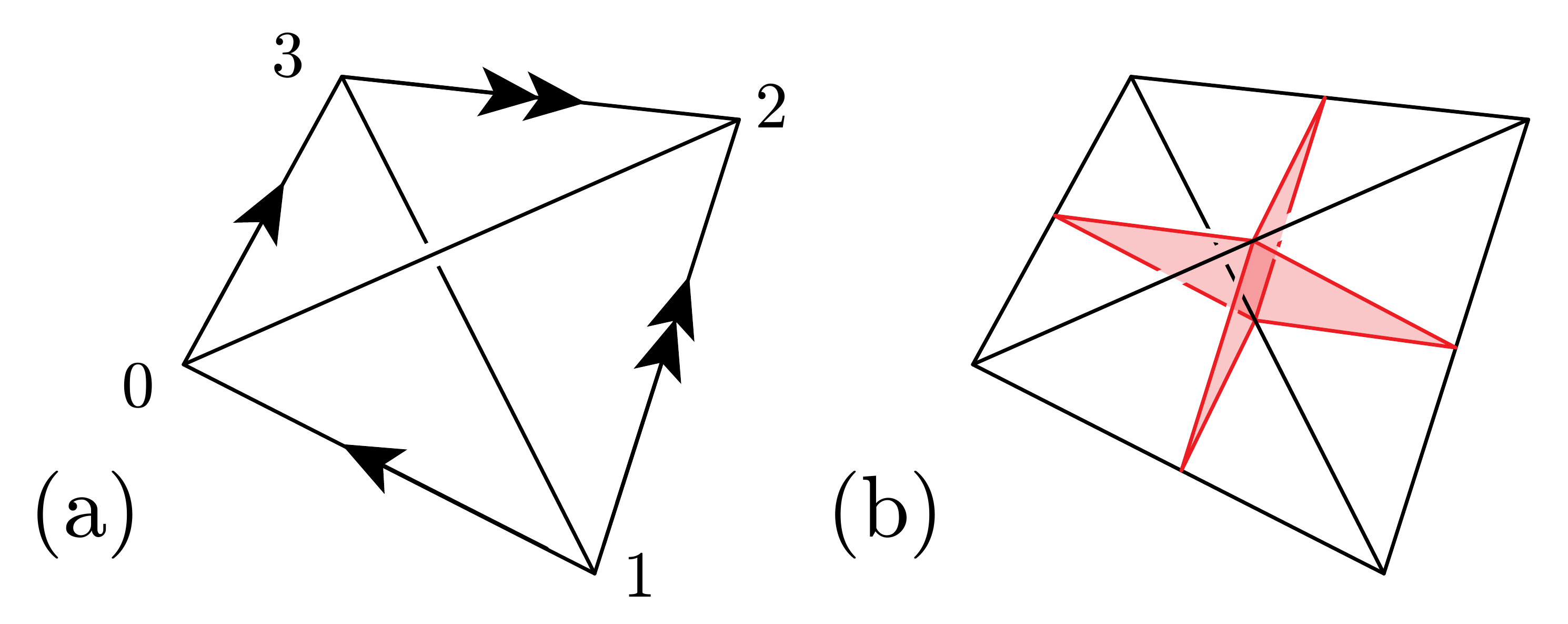}
\caption{An example vertical normal class in a single tetrahedron with appropriate identifications of edges. The quadrilaterals of the normal class are a single copy of each of the two vertical quadrilaterals within the tetrahedron. The vertex numbering shown matches Example \ref{4-gon in m136}.}
\label{4-gon_example.pdf}
\end{figure}

\begin{ex}\label{4-gon in m136}
We give an example of a triangulation $\tri_{\text{bad}}$ in which the situation shown in Figure \ref{4-gon_example.pdf} occurs.  We found $\tri_{\text{bad}}$ by modifying the triangulation of the manifold m136 given in the SnapPea census~\cite{snappea}. Our triangulation is a triangulation of 
m136, but is \emph{not} the triangulation given in the census. In fact, according to SnapPea, the triangulation given in the census is itself the canonical Epstein-Penner decomposition into polyhedra. Therefore $\tri_{\text{bad}}$ is not an example of subdividing the Epstein-Penner decomposition in a bad way, since there are no polyhedra that need to be subdivided into tetrahedra.\\

See Table \ref{table_example_gluings} for the gluing data for the 7 tetrahedra of $\tri_{\text{bad}}$, and Table \ref{table_example_edges} for the tetrahedra (with vertex numbers) glued around each of the 7 edges. Tetrahedron \#5 is the one whose edges are identified as in Figure \ref{4-gon_example.pdf}, and we can see this by noticing that it appears twice around edge \#2, as 5 (10) and 5 (03), and twice around edge \#4, as 5 (32) and 5 (12). This gives the edge identifications shown in Figure \ref{4-gon_example.pdf}. Using Regina \cite{regina} we can check that these quadrilaterals, together with a finite number of triangles (twenty) form an immersed, branched, closed normal surface.\\ 

Note that tetrahedron \#3 also has edges arranged as in Figure \ref{4-gon_example.pdf}. However, the corresponding quadrilaterals correspond to a spun-normal class rather than a closed normal class.\\

Also in Table \ref{table_example_gluings}, we list complex shape parameters for the tetrahedra of $\tri_{\text{bad}}$ which solve the gluing and completeness equations, found using Snap \cite{snap}. These parameters are for the edges (01) and (23) of each of the tetrahedra. The orientation convention is such that the edge (02) has shape parameter $\frac{1}{1-z}$  if the parameter for the edge (01) is $z$.
\end{ex}

\begin{table}[htb]
\caption{The gluing data for a triangulation of the manifold m136 from the SnapPea census containing a tetrahedron with identifications as in Figure \ref{4-gon_example.pdf}. To reconstruct the triangulation, take 7 tetrahedra with vertices labelled 0 through 3, all with consistent orientations, and make the appropriate gluings. For example, the top left entry in the table tells us to glue the face of tetrahedron \#0 with vertices labelled 0,1,2 to the face of tetrahedron \#1 with vertices labelled 3,1,2 in the orientation that matches the vertices in the order given.  \newline}
\centering

\begin{tabular}{ |r|cccc|r| }   
\hline
Tetrahedron  & Face 012 & Face 013 & Face 023 & Face 123 & Shape parameter\\ 
\hline
0 & 1 (312) & 4 (302) & 6 (130) & 4 (132) & $2i$\\
1 & 3 (102) & 2 (012) & 2 (203) & 0 (120) & $-1+2i$\\
2 & 1 (013) & 6 (321) & 1 (203) & 4 (031) & $\frac{3}{5} + \frac{1}{5}i$\\
3 & 1 (102) & 6 (230) & 5 (021) & 5 (023) & $-1$\\
4 & 5 (312) & 2 (132) & 0 (130) & 0 (132) & $\frac{1}{5} + \frac{2}{5} i$\\
5 & 3 (032) & 6 (012) & 3 (123) & 4 (120) & $2$\\
6 & 5 (013) & 0 (302) & 3 (301) & 2 (310) & $\frac{1}{2} + \frac{1}{2}  i$\\ [4pt] 
\hline
\end{tabular}

\label{table_example_gluings}
\end{table}

\begin{table}[htb]
\caption{The tetrahedra and vertex numbers of those tetrahedra incident to the edges of the triangulation given in Table \ref{table_example_gluings}. \newline}
\centering
\begin{tabular}{ |r|r|l| }   
\hline
Edge  & Degree &  Tetrahedron (Vertex numbers)\\ 
\hline
0 & 4 & 0 (01), 4 (30), 2 (21), 1 (31) \\
1 & 4 & 0 (02), 1 (32), 2 (30), 6 (13) \\
2 & 10 & 0 (03), 6 (10), 5 (10), 3 (30), 6 (02), 5 (03), 3 (13), 6 (30), 0 (23), 4 (32)\\
3 & 10 & 0 (12), 4 (13), 2 (32), 1 (30), 2 (20), 1 (02), 3 (12), 5 (02), 3 (02), 1 (12)\\
4 & 6 & 0 (13), 4 (02), 5 (32), 3 (32), 5 (12), 4 (12)\\
5 & 4 & 1 (01), 2 (01), 6 (32), 3 (10)\\
6 & 4 & 2 (13), 6 (21), 5 (31), 4 (01)\\
\hline
\end{tabular}
\label{table_example_edges}
\end{table}

\begin{rmk}
This example led us to the condition in Theorem \ref{no non periph H1 implies sas}. The square with edges identified as in Figure \ref{4-gon_example.pdf}(a) glues up to form a non-separating surface, and so a homology obstruction.
\end{rmk}

\begin{thm}
Suppose $M$ is a cusped hyperbolic 3-manifold with a geometric polyhedral decomposition consisting of polyhedra which have no $4$-gons or $6$-gons. Then $M$ has an ideal triangulation which admits a strict angle structure.
\end{thm}
\begin{proof}
Use Algorithm \ref{polyhedra cone vertices to triangulation} (with arbitrary choice of cone vertices) to produce a triangulation $\tri$. By Lemma \ref{quadrilaterals in only 4 and 6 gons} and Remark \ref{vertical only in pillows}, there are no vertical quadrilaterals in a vertical normal class $S\in \mathcal{N}(M;\tri)$. By Corollary \ref{cor: sas iff no vert}, $\tri$ admits a strict angle structure.
\end{proof}

\begin{rmk}
There are many ways in which this result can be extended, altering the hypothesis on the combinatorics of the geometric polyhedral decomposition. For example, there is a lot of choice in the cone vertices we use to triangulate the polyhedra.
If, for instance, each polyhedron has at most one 4-gon or 6-gon, then we can choose our cone vertices to not be on a 4-gon or 6-gon. Then the 4-gons and 6-gons are all bases of pyramids after coning, and so in Algorithm \ref{polyhedra cone vertices to triangulation} the triangulations match, we have no $n$-gonal pillows with $n=4$ or 6, and again we have no vertical quadrilaterals in a vertical normal class.
\end{rmk}

\section{Normal surfaces in PPP-cellulations}\label{sec: normal surfaces in PPP}

The aim of this section is to give topological conditions for the existence of a triangulation which admits a strict angle structure. To that end, we consider a particular class of surfaces, giving a notion of normality relative to a PPP-cellulation.\\ 

The types of normal disks we allow in each cell are as follows: In each polyhedron, we allow only vertex linking disks. In each $n$-gonal pillow, we allow vertex linking bigons, and when $n$ is even, twisted $n$-gons of two possible types. See Figure \ref{polyhedron_and_polygonal_pillow.pdf}. Each twisted $n$-gon has its vertices on the edges of the $n$-gonal pillow, and edges alternating on the top and bottom faces of the pillow. Whether the edge is on the top or the bottom at a given vertex distinguishes the two types of twisted $n$-gon disk. For our purposes, we will only have $n$-gonal pillows where $n$ is $4$ or $6$. \\

\begin{defn}
Suppose that $M$ is a manifold with a PPP-cellulation $\mathcal{P}'$. We define a subset  $\mathcal{N}(M;\mathcal{P}')$
of $(\Z_{\geq 0})^N$, where $N$ is the sum of the number of vertices in each polyhedron, plus the number of vertices in each polygonal pillow, plus two times the number of 4 or 6-sided polygonal pillows (one dimension for each allowed type of normal disk). $\mathcal{N}(M;\mathcal{P}')$ consists of vectors satisfying compatibility equations on the 2-cells of $\mathcal{P}'$ in the manner analogous to the normal surface compatibility equations for triangulations.
\end{defn}

\begin{figure}[htb]
\centering
\includegraphics[width=\textwidth]{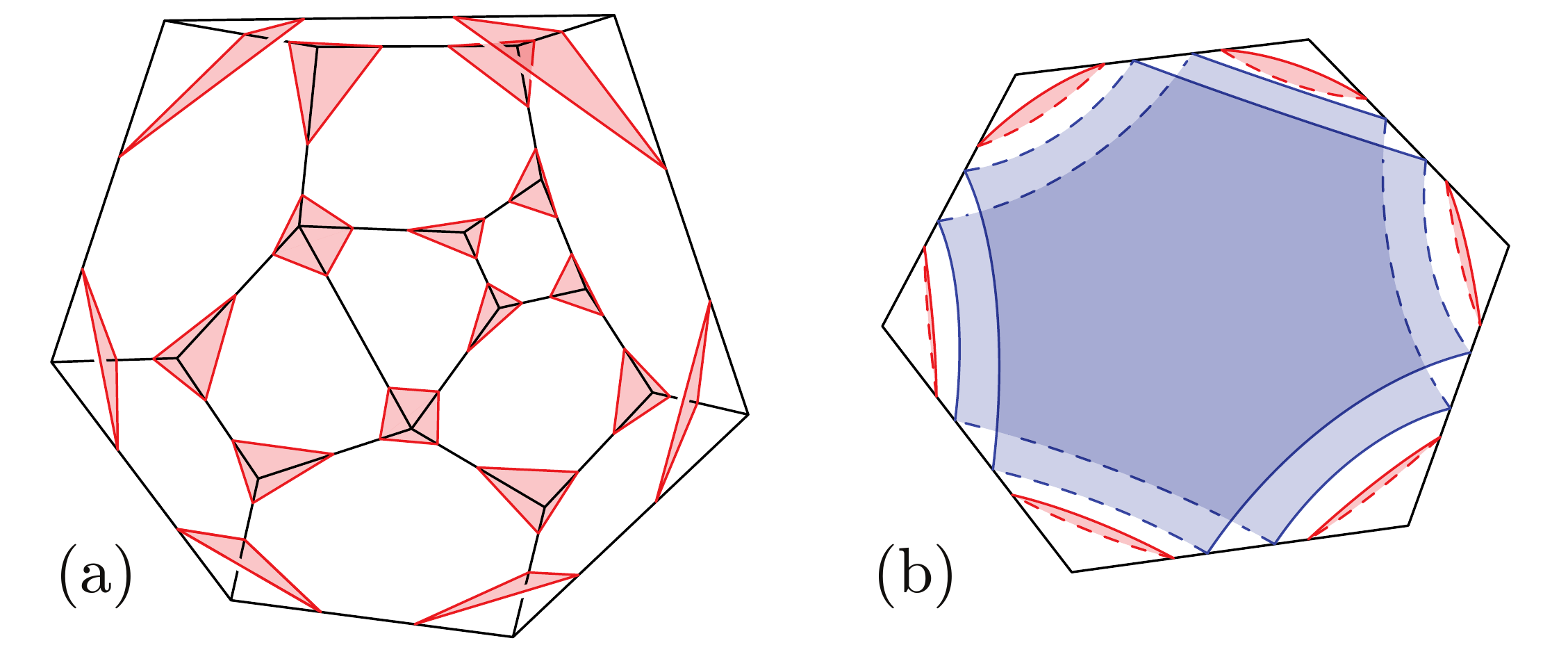}
\caption{In (a), a polyhedron with the allowed vertex linking normal disk types. In (b), a hexagonal pillow, with vertex linking normal bigons, and with two parallel copies of one of the two twisted hexagon normal disks. Arcs on the upper face of the pillow are drawn with solid lines, while those on the bottom are drawn dashed. }
\label{polyhedron_and_polygonal_pillow.pdf}
\end{figure}

Let $\mathcal{P}'$ be a PPP-cellulation of a manifold $M$, and $\tri$ a triangulation of $M$ where both $\mathcal{P}'$ and $\tri$ are formed from a geometric polyhedral decomposition $\mathcal{P}$ together with a compatible choice of cone vertices as in Algorithm \ref{polyhedra cone vertices to triangulation}. To each vertical $S \in \mathcal{N}(M;\tri)$ we associate a unique $S' \in  \mathcal{N}(M;\mathcal{P}')$ as follows: Within a polyhedron, all of the tetrahedra sharing a particular vertex have the same number of normal triangles at that vertex. We take this many normal disks linking the vertex in $S'$. For each polygonal pillow, Lemma \ref{change by +- lambda} tells us the possible arc patterns associated to $S$ on either side. The sign of the integer $\lambda$ given in this Lemma tells us which of the two types of twisted surface to use in $S'$, and the magnitude tells us how many. The remaining arc multiplicities tell us how many bigons to use at each vertex.\\

Note that given a vertical class $S \in \mathcal{N}(M;\tri)$, the corresponding $S'\in  \mathcal{N}(M;\mathcal{P}')$ gives a closed \emph{embedded} surface in $M$. If we can rule out the existence of these kinds of surfaces, then we rule out the existence of a vertical $S \in \mathcal{N}(M;\tri)$, and so show the existence of a strict angle structure.

\begin{proof}[of Theorem \ref{no non periph H1 implies sas}]
Let $M$ be a cusped hyperbolic 3-manifold homeomorphic to the interior of a compact 3-manifold $\overline{M}$ with torus or Klein bottle boundary components. The Epstein-Penner decomposition gives us a geometric polyhedral decomposition $\mathcal{P}$. We construct our triangulation $\tri$ and PPP-cellulation $\mathcal{P}'$ from $\mathcal{P}$ using Algorithm \ref{polyhedra cone vertices to triangulation}, but with a special choice of the cone vertices, as follows. The polyhedra of $\mathcal{P}$ together with their identifications determine a graph $\Gamma$ with one vertex for each polyhedron and an edge joining two vertices for each face incident to the two corresponding polyhedra.
Choose a maximal spanning tree $T$ for this graph. We will choose our cone vertices so that the 2-cells of $\mathcal{P}^{(2)}$ corresponding to the edges of $T$ have no polygonal pillows.\\ 

In order to do this, first choose the cone vertex for each polyhedron corresponding to a 
leaf\footnote{A leaf of a graph will mean a vertex of degree one.}
of $T$ to be any vertex not on the 2-cell corresponding to the edge of $T$ incident to the leaf. 
Then, remove the leaves and their incident edges from $T$ 
 and repeat, choosing the cone vertices on the new leaves in the same way. Eventually, the tree shrinks to either a single vertex, or the empty set (if the penultimate tree was a single edge). In the former case, we arbitrarily choose the cone vertex for the last polyhedron. 
 Let  $\mathcal{P}'$ denote the resulting PPP-cellulation.\\

Assume for contradiction that there is a vertical class $S\in \mathcal{N}(M;\tri)$, and therefore a corresponding $S' \in \mathcal{N}(M,\mathcal{P}')$. If $S'$ is not a fundamental surface (i.e.\thinspace $S'$ is a Haken sum) 
then we replace it with a summand which is. 
Then this fundamental surface is closed and connected, and we also require that it is not a peripheral torus or Klein bottle, i.e.\thinspace that it has some twisted disks.\\

We can embed the graph $\Gamma$ in $M$ so that each vertex is in the interior of the corresponding
polyhedron of $\mathcal{P}$ and each edge intersects the corresponding face of $\mathcal{P}$ transversely in one point. Then we can isotope $S'$ so that the bigons and vertex linking disks in $S'$ are disjoint from $\Gamma$ and each twisted disk in $S'$ intersects $\Gamma$ transversely in one point. Further, by the construction of $\mathcal{P}'$, $S'$ is then disjoint from the spanning tree $T \subset \Gamma$.\\

Assume that some twisted disk has odd multiplicity in $S'$ so meets an edge $e$ of $\Gamma$ 
transversely in an odd number of points. Then $T \cup e$ contains a simple closed curve $\gamma$
that has odd intersection number with $S'$, so represents a non-trivial element of $H_1(\overline{M}; \Z_2)$. 
This element is not in the image of the map from $H_1(\partial \overline{M}; \Z_2)$ in the long exact sequence 
$$\cdots \to H_2(\overline{M}, \partial \overline{M}; \Z_2)  \to H_1(\partial \overline{M}; \Z_2) \to H_1(\overline{M}; \Z_2) \to H_1(\overline{M}, \partial \overline{M}; \Z_2) \to \cdots $$
because no sum of peripheral loops can have odd intersection with a closed surface. Therefore $\gamma$ has non-zero image under the map $H_1(\overline{M}; \Z_2) \to H_1(\overline{M}, \partial \overline{M}; \Z_2)$, giving a contradiction.\\

So, assume that all the twisted disks have even multiplicity in our fundamental surface $S'$. If all the bigons and vertex linking disks also have even multiplicity, then the whole surface is a double of another surface so is not fundamental. \\

Therefore we may assume without loss of generality that at least one bigon or vertex linking disk has odd multiplicity. 
Consider the set of normal disks $S''$ consisting of one copy of each bigon or vertex linking disk with odd multiplicity. We claim that this set of normal disks satisfies the normal surface compatibility equations. To see this, consider the compatibility equation for a given normal arc $\alpha$ in a face $F$ of the PPP-cellulation. Three types of normal disk contribute to the multiplicity for $\alpha$: one type of vertex linking disk in the polyhedron incident to $F$, one type of bigon in the polygonal pillow incident to $F$, and one type of twisted disk in that polygonal pillow. Since we have an even number of twisted disks at $\alpha$, the numbers of bigons and vertex linking disks at $\alpha$ must have the same parity in order to satisfy the compatibility equation for $\alpha$. Thus, if $S''$ has any disks at $\alpha$, it must have one vertex linking disk and one bigon, and since these are incident to $\alpha$ from opposite sides of $F$, $S''$ satisfies the compatibility equation at $\alpha$.\\

So our surface $S'$ is a Haken sum of $S''$ with some surface containing at least one twisted disk. Once again, $S'$ is not fundamental and the proof is complete. \end{proof}

\begin{rmk} \label{obstruction} The homological obstruction to finding a strict angle structure arising in the above proof can be described explicitly as follows: Given a  vertical normal class $S\in \mathcal{N}(M;\tri)$, let $S_{bad}$ be the union of all the faces in $\mathcal{P}$ corresponding to polygonal pillows in $\mathcal{P}'$ where the invariant $\lambda$ from Lemma \ref{change by +- lambda} is odd. (Recall that $\lambda$ describes the difference in multiplicities of vertex linking arcs in the arc patterns of $S$ on the top and bottom of the pillow.) \\

Then $S_{bad}$ represents a non-trivial relative homology class 
$[S_{bad}] \in H_2(\overline{M},\partial \overline{M}; \Z_2)$, 
and is the image of the homology class $[S'] \in H_2(\overline{M};\Z_2)$ of the closed surface $S'$ constructed above under the natural homomorphism 
$ H_2(\overline{M};\Z_2) \to H_2(\overline{M},\partial \overline{M}; \Z_2)$. Further, by Lemma \ref{quadrilaterals in only 4 and 6 gons}, 
we see that $[S_{bad}]$  is represented by a sum of $4$-gons and $6$-gons in $\mathcal{P}$.\\

Alternatively, by Lefschetz duality, we can think of the obstruction as a cohomology class  $[S_{bad}]^* \in H^1(\overline{M};\Z_2)$. This can be regarded as the homomorphism $H_1(\overline{M};\Z_2) \to \Z_2$,  vanishing on the image of $H_1(\partial \overline{M};\Z_2)$, given by taking the mod 2 intersection number with $[S_{bad}]$. 
\end{rmk}

\begin{rmk}\label{not_geometric}
Although our construction gives triangulations which admit strict angle structures, unfortunately they will not generally be geometric. In particular, after we cone the polyhedra obtained from the geometric polyhedral decomposition, if any triangulations of faces do not match then the triangulation of the manifold we construct will not be geometric. The reason for this is that we know what the shapes of the tetrahedra will be in the complete hyperbolic structure: they are determined by the shapes of the polyhedra in the geometric polyhedral decomposition. In particular, a tetrahedron inserted to bridge between different triangulations of a polygon will be flat in $\mathbb{H}^3$, since they are contained within the ideal hyperbolic polygon. \end{rmk}

\begin{ex}\label{strict angle struct not geometric}
We give an example of a triangulation demonstrating the features described in Remark \ref{not_geometric}: it is not geometric but it admits a strict angle structure. Our manifold $M$ is the complement of the Turk's head knot, shown in Figure \ref{turks_head_knot.pdf}. Proposition I.2.5 of \cite{sakumaweeks95} tells us that the canonical decomposition of $M$ consists of two square anti-prisms, and is the same as the decomposition described in the proof of Theorem \ref{alternating link}. See Figure \ref{turks_head_knot_polyhedra.pdf}. Each polyhedron has two square faces and eight triangular faces, and each face is glued to a corresponding face of the other polyhedron. There are 8 edges after the gluings, which are numbered on the diagrams, and these labels determine the gluing.\\

\begin{figure}[htb]
\centering
\includegraphics[width=\textwidth]{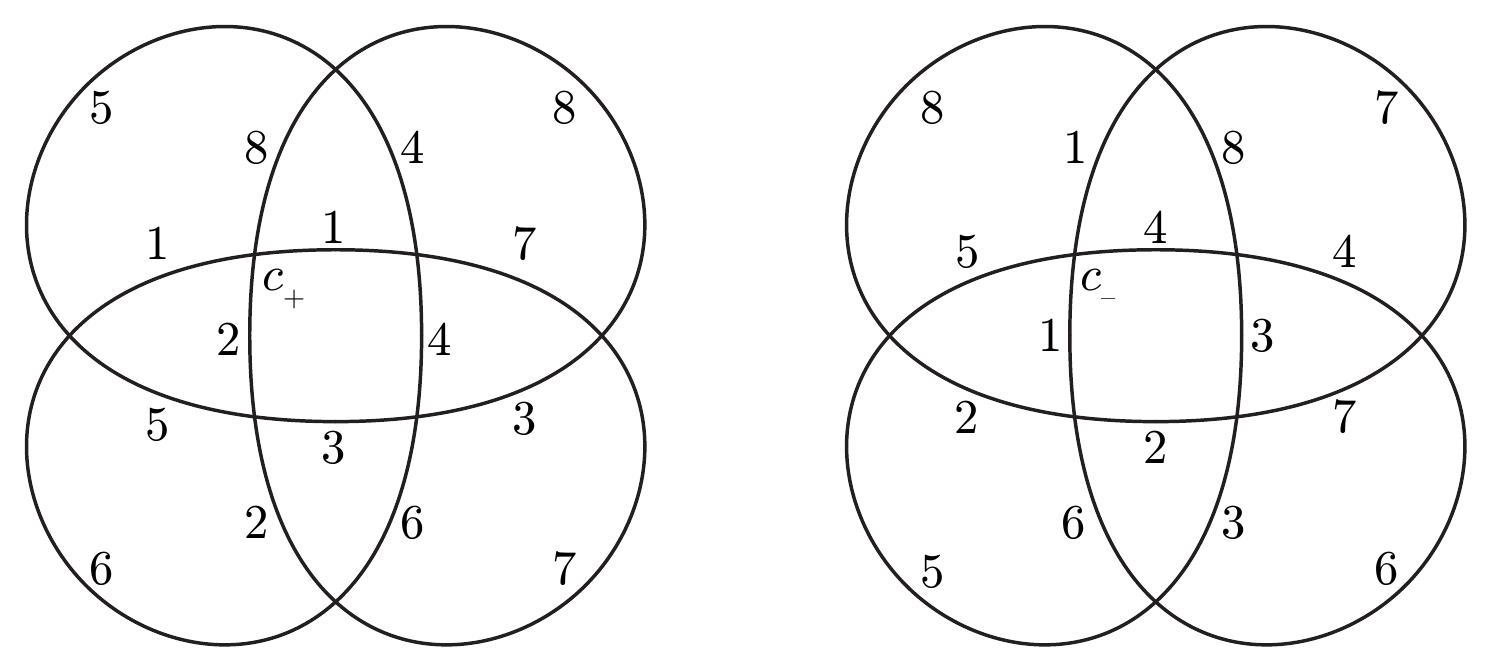}
\caption{The gluing combinatorics for the two square antiprisms that make up the canonical decomposition of the complement of the Turk's head knot. The left diagram shows the boundary of the polyhedron above the plane of the knot diagram, and shows the boundary pattern from inside the polyhedron. The right diagram shows the polyhedron below the plane, and shows boundary pattern from outside the polyhedron. }
\label{turks_head_knot_polyhedra.pdf}
\end{figure}

If we apply Algorithm \ref{polyhedra cone vertices to triangulation} with this geometric polyhedral decomposition and choice of cone vertices the vertices marked $c_+$ and $c_-$, then the resulting PPP-cellulation has a single 4-gonal pillow corresponding to the central square in the two diagrams of Figure \ref{turks_head_knot_polyhedra.pdf}, which has as layered triangulation a single tetrahedron $\sigma$. The resulting triangulation is not geometric, since $\sigma$ is flat at the complete structure. However, it does have a strict angle structure, by the following argument. First, by Corollary \ref{cor: sas iff no vert}, a strict angle structure exists if a non-trivial vertical normal class does not. By Remark \ref{vertical only in pillows}, we can only have vertical quadrilaterals in $\sigma$, since all other tetrahedra have positive angles. However, the Q-matching equations would have to hold for such a vertical class. From Figure \ref{turks_head_knot_polyhedra.pdf} we see that there are no identifications of the edges of $\sigma$ (they are the edges labelled 1, 2, 3 and 4, together with two edges corresponding to the diagonals of the 4-gon, whose other incident tetrahedra are inside the two respective polyhedra). This means that the Q-matching equations are impossible to satisfy in this situation, if we use a non-zero number of only vertical quadrilaterals. Therefore the only possible vertical class is trivial, and so by Corollary \ref{cor: sas iff no vert} we have the existence of a strict angle structure. \\

Note that this construction relies on a particularly `bad' choice of cone vertices. We can also choose them in such a way that the resulting triangulation is geometric.
\end{ex}

\section{Generalisations}\label{generalisations}

\subsection{Using combinatorial polyhedra}

Suppose $M$ is the interior of a compact 3-manifold with torus or Klein bottle boundary components and that $M$  is irreducible and atoroidal.
Assume we are given a method to divide $M$ into a collection of ideal polyhedra, each of which is assigned dihedral angles satisfying the conditions for Rivin's generalisation of Andreev's theorem, Theorem \ref{rivin_andreev}.\\

One can then also require that the sum of dihedral angles around each edge is $2\pi$, giving a notion of a polyhedral angle structure. We then apply Rivin's theorem to each polyhedron so that it has a convex ideal hyperbolic structure. 
This means that we can divide each polyhedron by coning from a vertex and get a
solution of the angle equations with flat regions, exactly as for the
geometric polyhedral decomposition case. In particular, we can proceed as before and get strict angle
structures on the tetrahedra, assuming that the homological obstruction in Theorem \ref{no non periph H1 implies sas} vanishes. \\

Note that in the above, we are not interested in the question as to whether the convex ideal hyperbolic polyhedra glue together by isometries in $\H^3$, just the angle conditions.

\begin{rmk}
Section 4 is an example of this approach. It would be interesting to find other ways of generating such polyhedral decompositions. In particular, knowing we have a complete hyperbolic structure
on $M$ allows us to use Epstein-Penner to obtain such a good decomposition, but it is reasonable to ask if there are direct methods of 
finding decompositions with the above properties. 
\end{rmk}

\begin{rmk}
A similar strategy may also work using the more general angled blocks of Futer and Gu\'{e}ritaud \cite{futer_gueritaud_angled_blocks}. However, since these blocks are not necessarily convex (or even simply connected), the coning construction would not be enough to subdivide them into ideal tetrahedra.
\end{rmk}

\subsection{Dealing with non-cone triangulations of polygons}

For the results of this paper, we only needed to use the coning construction to subdivide our polyhedra, but there are many other ways of subdividing a polyhedron into tetrahedra without introducing additional vertices.  In particular, the construction of a layered triangulation in a polygonal pillow of Section \ref{sec: layered triangs} assumes that the triangulations of the polygon on either side of the polygonal pillow are both cone triangulations. However, other constructions may not have this feature. With a more complicated version of Algorithm \ref{layered triang cone triangulations}, it is possible to bridge between arbitrary triangulations of a polygon, with properties along the lines of Lemma \ref{quadrilaterals in only 4 and 6 gons}. For example, the following lemma (which we state without proof) was part of an early version of our approach to this problem.

\begin{lemma}\label{no change in polygon}
Let $t_-$ and $t_+$ be triangulations of the $n$-gon. Then there exists a layered triangulation of $Q$, with triangulations $t_-$ and $t_+$ on the two sides, and with the following property:

Suppose we have a vertical normal class $S$ in a triangulation of a PPP-cellulation $\mathcal{P}$ of a manifold, with $Q$ an $n$-gonal pillow in $\mathcal{P}$ with the given layered triangulation. If $S$ has the same arc pattern on the two sides of $Q$, each pattern made up of vertex linking arcs with some multiplicity at each vertex, then the arc pattern is the same on all layers of $Q$, and so $Q$ contains no vertical quadrilaterals of $S$.
\end{lemma}
Note that the arc pattern on a side of $Q$ makes sense as a set of disjoint curves on a polygon, rather than the triangulated polygon. It is in this sense that we mean that the arc pattern is the same on the two sides of $Q$, despite the triangulations being different.

\begin{rmk}
As an example of how Lemma \ref{no change in polygon} could be useful, suppose that we have a geometric polyhedral decomposition of our manifold, and that we can subdivide the polyhedra in such a way that the triangulations of polygonal faces fail to match their counterparts only in the case of polygons with an odd number of sides. (We also need that the resulting ideal hyperbolic tetrahedra have positive volume.) This is possibly easier to achieve using arbitrary subdivisions of our polyhedra, rather than only using coning.  Lemma \ref{no change in polygon} then gives layered triangulations of the polygonal pillows needed to bridge between non-matching triangulations of polygonal faces. We then ask if there are any non-trivial vertical normal classes, as in Corollary \ref{cor: sas iff no vert}. By Lemma \ref{change by +- lambda}, we know that for such a class and for each (by assumption, odd sided) polygonal pillow, the patterns of arcs on the top and bottom copies of the polygon would be identical, each pattern made up of vertex linking arcs with some multiplicity at each vertex. Lemma \ref{no change in polygon} then says that the vertical class would be entirely made up of triangles in each polygonal pillow, and so we get the existence of a strict angle structure.
\end{rmk}

\end{document}